\documentclass[12pt]{amsart}
\usepackage[nobysame,abbrev,alphabetic]{amsrefs}
\usepackage{amssymb}
\usepackage[all]{xy}

\DeclareMathOperator{\Gal}{Gal}

\DeclareMathOperator{\id}{id}
\DeclareMathOperator{\Img}{Im}

\DeclareMathOperator{\Lyn}{Lyn}

\DeclareMathOperator{\res}{res}

\DeclareMathOperator{\Sh}{Sh}

\DeclareMathOperator{\trg}{trg}

\DeclareFontFamily{U}{wncy}{}
\DeclareFontShape{U}{wncy}{m}{n}{<->wncyr10}{}
\DeclareSymbolFont{mcy}{U}{wncy}{m}{n}
\DeclareMathSymbol{\Sha}{\mathord}{mcy}{"58}
\DeclareMathSymbol{\sha}{\mathord}{mcy}{"78}

\begin{document}

\newtheorem{thm}{Theorem}[section]
\newtheorem{cor}[thm]{Corollary}
\newtheorem{lem}[thm]{Lemma}
\newtheorem{prop}[thm]{Proposition}
\newtheorem{defin}[thm]{Definition}
\newtheorem{exam}[thm]{Example}
\newtheorem{examples}[thm]{Examples}
\newtheorem{rem}[thm]{Remark}
\newtheorem{case}{\sl Case}
\newtheorem{claim}{Claim}
\newtheorem{prt}{Part}
\newtheorem*{mainthm}{Main Theorem}
\newtheorem*{thmA}{Theorem A}
\newtheorem*{thmB}{Theorem B}
\newtheorem*{thmC}{Theorem C}
\newtheorem*{thmD}{Theorem D}
\newtheorem{question}[thm]{Question}
\newtheorem*{notation}{Notation}
\swapnumbers
\newtheorem{rems}[thm]{Remarks}
\newtheorem*{acknowledgment}{Acknowledgment}

\newtheorem{questions}[thm]{Questions}
\numberwithin{equation}{section}

\newcommand{\Bock}{\mathrm{Bock}}
\newcommand{\dec}{\mathrm{dec}}
\newcommand{\diam}{\mathrm{diam}}
\newcommand{\dirlim}{\varinjlim}
\newcommand{\discup}{\ \ensuremath{\mathaccent\cdot\cup}}
\newcommand{\divis}{\mathrm{div}}
\newcommand{\gr}{\mathrm{gr}}
\newcommand{\nek}{,\ldots,}
\newcommand{\indec}{{\rm indec}}
\newcommand{\Ind}{\mathrm{Ind}}
\newcommand{\inv}{^{-1}}
\newcommand{\isom}{\cong}
\newcommand{\Massey}{\mathrm{Massey}}
\newcommand{\ndiv}{\hbox{$\,\not|\,$}}
\newcommand{\nil}{\mathrm{nil}}
\newcommand{\pr}{\mathrm{pr}}
\newcommand{\sep}{\mathrm{sep}}
\newcommand{\sh}{\mathrm{sh}}
\newcommand{\SR}{\mathrm{SR}}
\newcommand{\tagg}{^{''}}
\newcommand{\tensor}{\otimes}
\newcommand{\alp}{\alpha}
\newcommand{\gam}{\gamma}
\newcommand{\Gam}{\Gamma}
\newcommand{\del}{\delta}
\newcommand{\Del}{\Delta}
\newcommand{\eps}{\epsilon}
\newcommand{\lam}{\lambda}
\newcommand{\Lam}{\Lambda}
\newcommand{\sig}{\sigma}
\newcommand{\Sig}{\Sigma}
\newcommand{\bfA}{\mathbf{A}}
\newcommand{\bfB}{\mathbf{B}}
\newcommand{\bfC}{\mathbf{C}}
\newcommand{\bfF}{\mathbf{F}}
\newcommand{\bfP}{\mathbf{P}}
\newcommand{\bfQ}{\mathbf{Q}}
\newcommand{\bfR}{\mathbf{R}}
\newcommand{\bfS}{\mathbf{S}}
\newcommand{\bfT}{\mathbf{T}}
\newcommand{\bfZ}{\mathbf{Z}}
\newcommand{\dbA}{\mathbb{A}}
\newcommand{\dbC}{\mathbb{C}}
\newcommand{\dbF}{\mathbb{F}}
\newcommand{\dbN}{\mathbb{N}}
\newcommand{\dbQ}{\mathbb{Q}}
\newcommand{\dbR}{\mathbb{R}}
\newcommand{\dbU}{\mathbb{U}}
\newcommand{\dbZ}{\mathbb{Z}}
\newcommand{\grf}{\mathfrak{f}}
\newcommand{\gra}{\mathfrak{a}}
\newcommand{\grA}{\mathfrak{A}}
\newcommand{\grB}{\mathfrak{B}}
\newcommand{\grh}{\mathfrak{h}}
\newcommand{\grI}{\mathfrak{I}}
\newcommand{\grL}{\mathfrak{L}}
\newcommand{\grm}{\mathfrak{m}}
\newcommand{\grp}{\mathfrak{p}}
\newcommand{\grq}{\mathfrak{q}}
\newcommand{\grR}{\mathfrak{R}}
\newcommand{\calA}{\mathcal{A}}
\newcommand{\calB}{\mathcal{B}}
\newcommand{\calC}{\mathcal{C}}
\newcommand{\calE}{\mathcal{E}}
\newcommand{\calG}{\mathcal{G}}
\newcommand{\calH}{\mathcal{H}}
\newcommand{\calK}{\mathcal{K}}
\newcommand{\calL}{\mathcal{L}}
\newcommand{\calM}{\mathcal{M}}
\newcommand{\calW}{\mathcal{W}}
\newcommand{\calV}{\mathcal{V}}

\title[The lower $p$-central series and shuffle relations]
{The lower $p$-central series of a free profinite group and the shuffle algebra}

\author{Ido Efrat}
\address{Mathematics Department\\
Ben-Gurion University of the Negev\\
P.O.\ Box 653, Be'er-Sheva 84105\\
Israel} \email{efrat@math.bgu.ac.il}


\maketitle
\begin{abstract}
For a prime number $p$ and a free profinite group $S$ on the basis $X$, let $S^{(n,p)}$, $n=1,2\nek$ be the  lower $p$-central filtration of $S$.
For $p>n$, we give a combinatorial description of $H^2(S/S^{(n,p)},\dbZ/p)$ in terms of the Shuffle algebra on $X$.
\end{abstract}

\section{Introduction}
For a fixed prime number $p$ and a profinite group $G$, the \textsl{lower $p$-central filtration} $G^{(n,p)}$, $n=1,2,3\nek$ of $G$ is defined inductively by
\[
G^{(1,p)}=G,\quad G^{(n+1,p)}=(G^{(n,p)})^p[G^{(n,p)},G].
\]
Thus $G^{(n+1,p)}$ is the closed subgroup of $G$ generated by all powers $g^p$ and commutators $[g,h]=g\inv h\inv gh$, with $g\in G^{(n,p)}$ and $h\in G$.
The subgroups $G^{(n,p)}$ are normal in $G$, and we denote $G^{[n,p]}=G/G^{(n,p)}$.
We write $H^i(G)=H^i(G,\dbZ/p)$ for the profinite cohomology group of $G$ with respect to its trivial action on $\dbZ/p$.

In this paper we consider a free profinite group $S$ on the basis $X$.
For $p>n$, we give an explicit combinatorial description of  $H^2(S^{[n,p]})$ in terms of the \textsl{shuffle algebra} $\Sh_\dbZ(X)$ over $X$.
To state it, denote the free monoid  on $X$ by $X^*$, and consider its elements as associative words in the alphabet $X$.
Let $\dbZ\langle X\rangle$ be the free $\dbZ$-module over the basis $X^*$.
The {\sl shuffle product} of words $u=(x_1\cdots x_r)$ and $v=(x_{r+1}\cdots x_{r+t})$ is defined by
\[
u\sha v=\sum_\sig(x_{\sig(1)}\cdots x_{\sig(r+t)})\in\dbZ\langle X\rangle,
\]
where $\sig$ ranges over all permutations of $1,2\nek r+t$ such that $\sig(1)<\cdots<\sig(r)$ and $\sig(r+1)<\cdots<\sig(r+t)$.
This extends by linearity to an associative and commutative bilinear map $\sha$ on $\dbZ\langle X\rangle$, and
the shuffle algebra $\Sh_\dbZ(X)$ is the resulting $\dbZ$-algebra.

Next let $M$ be the submodule of $\dbZ\langle X\rangle$ generated by all products $u\sha v$ such that $u,v\in X^*$ are nonempty words.
The $\dbZ$-module $\Sh_\dbZ(X)_{\rm indec}=\dbZ\langle X\rangle/M$ is called the {\sl indecomposable quotient} of $\Sh_\dbZ(X)$.
The $\dbZ$-module $\Sh_\dbZ(X)$, and therefore also it quotient $\Sh_\dbZ(X)_{\rm indec}$, are graded by the length $|w|$ of words $w$.
Our main result states that, for $p$ sufficiently large, all relations in $H^2(S^{[n,p]})$  are formal consequences of shuffle relations:

\begin{mainthm}
\label{main theorem}
For every integer $n\geq2$ and a prime number $p>n$ there is a canonical isomorphism
\[
\Bigl(\bigoplus_{s=1}^n\Sh_\dbZ(X)_{{\rm indec},s}\Bigr)\tensor(\dbZ/p)\xrightarrow{\ \ \sim\ \
}H^2(S^{[n,p]}).
\]
\end{mainthm}

This isomorphism is based on a general construction of cohomological elements $\rho^*\gam_{n,s}$ in $H^2(G)$, where $G$ is an arbitrary profinite group.
These elements are constructed as certain pullbacks by unipotent upper-triangular representations $\rho$ of $G$, and we refer to \S\ref{section on pullbacks} for their precise definition.
They include Bockstein elements, cup products, and more generally, elements of Massey products in $H^2(G)$ (Examples \ref{example Bockstein} and \ref{example Massey}).
The case where $G=S^{[n,p]}$, for $S$ a free profinite group, is of a generic nature, but the cohomological elements $\rho^*\gam_{n,s}$ are interesting also for profinite groups $G$ of \textsl{arithmetical} origin.
In particular, in \S\ref{section on Galois groups} we take $G$ to be the absolute Galois group of a field containing a root of unity of order $p$.
Then the Merkurjev--Suslin theorem implies that each element of  $H^2(G)$ decomposes as a sum of cup products.
However the decomposition pattern of the $\rho^*\gam_{n,s}$ is known explicitly only in very few cases, and these decompositions correspond to deep facts in Galois cohomology, which in turn provide restrictions on the possible group-theoretic structure of $G$ (see \S \ref{section on Galois groups}).

The proof of Theorem \ref{main theorem} leans heavily on our previous paper \cite{Efrat17}, where we construct for every word $w\in X^*$ of length $\leq n$  an element $\alp_{w,n}$ of $H^2(S^{[n,p]})$, and show that the elements $\alp_{w,n}$, where $w$ is a {\sl Lyndon word}, form a linear basis of $H^2(S^{[n,p]})$ (see \S\S\ref{section on words}-\ref{section on shuffle relations} for details).
The elements $\alp_{w,n}$ are special instances of the above-mentioned pullbacks $\rho^*\gam_{n,s}$, where the representation $\rho=\bar\rho^w$ is constructed in this case using \textsl{Magnus theory} (see \S\ref{section on shuffle relations}).
The existence of an \textsl{epimorphism} $\Bigl(\bigoplus_{s=1}^n\Sh_\dbZ(X)_{\indec,s}\Bigr)\tensor(\dbZ/p)\to H^2(S^{[n,p]})$ is a formal consequence of the results of \cite{Efrat17}.
It was further shown there, by an explicit computation, that this epimorphism is an isomorphism for $n=2,p>2$ and for $n=3,p>3$.
Thus the present paper settles this problem in general, by proving the injectivity of the epimorphism for $p>n$.
Note that for $n=p$ it need not be injective -- see Example \ref{counterexamples}.
The main new ingredient in the current paper, which leads to the proof of the injectivity, is the use of the Radford polynomials (discussed in \S\ref{section on words}).

When $n=2$, the quotient $S^{[2,p]}$ is an elementary abelian $p$-group.
Then the Main Theorem recovers the known fact that for an elementary abelian $p$-group $\bar S$, with $p$ odd, one has
$H^1(\bar S)\oplus\bigwedge^2 H^1(\bar S) \xrightarrow{\sim}H^2(\bar S)$ via the map which is the Bockstein homomorphism on the first direct summand, and is induced by the cup product on the second summand (\cite{EfratMinac11}*{\S2D)}, \cite{Topaz16}*{Fact 8.1}).
In this sense our work traces back to the pioneering work of Labute on the duality between $S^{(2,p)}/S^{(3,p)}$ and the second cohomology \cite{Labute67}.

I thank Claudio Quadrelli, Jacob Stix, and Adam Topaz for discussions related to this work.

\section{Upper-triangular unipotent matrices}
\label{section on unitriangular matrices}
For a unital commutative ring $R$, let $\dbU_s(R)$ denote the group of all upper-triangular unipotent $(s+1)\times(s+1)$-matrices over $R$.

We consider integers $n,s$ such that $n\geq2$ and $1\leq s\leq n$.
Let $R=\dbZ/p^{n-s+1}$  and $\dbU=\dbU_s(R)$.
We write $I_{s+1}$ for the unit matrix in $\dbU$, and $E_{i,j}$ for the $(s+1)\times(s+1)$-matrix over $R$ which is $1$ at entry $(i,j)$, and is $0$ elsewhere.

\begin{lem}
\label{filtration of U}
The lower $p$-central filtration of $\dbU$ satisfies
\begin{enumerate}
\item[(a)]
$\dbU^{(n,p)}\isom\dbZ/p$;
\item[(b)]
$\dbU^{(n,p)}$ lies in the center of $\dbU$;
\item[(c)]
$\dbU^{(n+1,p)}=1$.  % [needed?]
\end{enumerate}
\end{lem}
\begin{proof}
(a) \quad
By  \cite{Efrat17}*{Prop.\ 6.3(a)}, $\dbU^{(n,p)}=I_{s+1}+\dbZ p^{n-s}E_{1,s+1}$.
Therefore there is an isomorphism
\[
\dbU^{(n,p)}\to\dbZ/p, \quad  I_{s+1}+kp^{n-s}E_{1,s+1}\mapsto k\pmod p.
\]

(b) \quad
See \cite{Efrat17}*{Prop.\ 6.3(b)}.

\medskip

(c) \quad
By  (a) and (b), $\dbU^{(n+1,p)}=(\dbU^{(n,p)})^p[\dbU^{(n,p)},\dbU]=1$.
\end{proof}

We obtain a central extension
\begin{equation}
\label{exact sequence for U}
0\to\dbZ/p\to\dbU\xrightarrow{\lam}\dbU^{[n,p]}\to1
\end{equation}
where $\lam$ is the canonical epimorphism.
This extension corresponds to a cohomology element $\gam_{n,s}\in H^2(\dbU^{[n,p]})$, in the sense of the Schreier correspondence (see e.g., \cite{NeukirchSchmidtWingberg}*{Th.\ 1.2.4}).

Next we compute the exponent of $\dbU$, based on an argument by W.\ Sawin \cite{Sawin18}.
First we record a few elementary facts about binomial coefficients.

\begin{prop}
\label{binomial coef}
\begin{enumerate}
\item[(a)]
Let $l,k$ be integers such that $1\leq l\leq p^k$.
Then $p^k/\gcd(p^k,l)$ divides $\binom {p^k}l$.
\item[(b)]
Let $l,k$ be integers such that $1\leq l\leq p^k$ and $l$ is a $p$-power.
Then $p^k/l$ is the maximal $p$-power dividing $\binom {p^k}l$.
\item[(c)]
Let $t,k$ be integers such that $1\leq t\leq p^k$, and let $q$ be a $p$-power.
Then  $q|\binom{p^k}l$, $l=1,2\nek t$, if and only if $qp^{\lfloor\log_pt\rfloor}|p^k$.
\item[(d)]
For $m\geq1$, one has $p|\binom ml$, $l=1,2\nek m-1$, if and only if $m$ is a $p$-power.
\end{enumerate}
\end{prop}
\begin{proof}
(a) \quad
Let $v_p$ be the $p$-adic valuation.
As $\binom {p^k}l=\frac{p^k}l\binom{p^k-1}{l-1}$, we have $v_p\bigl(\binom {p^k}l\bigr)\geq k-v_p(l)=v_p(p^k/\gcd(p^k,l))$, and the assertion follows.

\medskip

(b) \quad
See \cite{RotaSagan80}*{Cor.\ 4.3}.

\medskip

(c) \quad
Suppose that $qp^{\lfloor\log_pt\rfloor}|p^k$.
Since $p^{\lfloor\log_pt\rfloor}$ is the largest $p$-power $\leq t$, for every $1\leq l\leq t$ we have $q\gcd(p^k,l)|p^k$, so by (a), $q|\binom{p^k}l$.

Conversely, if $q$ divides $\binom{p^k}l$ for $l=p^{\lfloor\log_pt\rfloor}$, then by (b), $q$ divides $p^k/l$.

\medskip

(d) \quad
See e.g., \cite{Fine47}.
Note that the ``if" part is contained in (c).
\end{proof}

\begin{prop}[Sawin]
\label{exponent}
Let $s$ be a positive integer and let $q>1$ be a $p$-power.
Then the group $\dbU=\dbU_s(\dbZ/q)$ has exponent $qp^{\lfloor\log_ps\rfloor}$.
\end{prop}
\begin{proof}
Let $p^r=qp^{\lfloor\log_ps\rfloor}$ and observe that $p^r>s$.

Consider an arbitrary matrix $I_{s+1}+N$ in  $\dbU$, where $N$ is upper-triangular with zeros on the main diagonal.
Thus $N^{s+1}=0$.
For every $k\geq0$ and for $t=\min(s,p^k)$, the binomial expansion formula implies that
\begin{equation}
\label{binom}
(I_{s+1}+N)^{p^k}=\sum_{l=0}^{p^k}\binom {p^k}lN^l=\sum_{l=0}^t\binom {p^k}lN^l.
\end{equation}

For $k=r$ we have $t=s$, and Proposition \ref{binomial coef}(c) implies that $(I_{s+1}+N)^{p^r}=I_{s+1}$.
Therefore the exponent of $\dbU$ divides  $p^r$.

On the other hand, take $k=r-1$.
As neither $qp^{\lfloor\log_ps\rfloor}$ nor $qp^k$ divides $p^k=p^{r-1}$, we have $qp^{\lfloor\log_pt\rfloor}\ndiv p^k$.
Proposition \ref{binomial coef}(c) therefore yields $1\leq l\leq t\,(\leq s)$ such  that $q\ndiv\binom{p^k}l$.
Take $N$ to be the $(s+1)\times(s+1)$-matrix which is $1$ on the super-diagonal and is $0$ elsewhere.
Then for $1\leq i\leq j\leq s+1$ we have $(N^m)_{ij}=1$, if $j-i=m$, and $(N^m)_{ij}=0$ otherwise.
It follows from (\ref{binom}) that $(I_{s+1}+N)^{p^{r-1}}\neq I_{s+1}$.
Consequently, the exponent of $\dbU$ is exactly $p^r$.
\end{proof}

\section{Pullbacks}
\label{section on pullbacks}
As before, let $n,s$ be integers such that $1\leq s\leq n$ and $n\geq2$, and set $\dbU=\dbU_s(R)$ where $R=\dbZ/p^{n-s+1}$.
Let $\pr_{ij}\colon\dbU\to R$ denote the projection on the $(i,j)$-entry, $i\leq j$.

\begin{defin}
\label{basic definitions}
\rm
Given a profinite group $G$ and a continuous homomorphism $\rho\colon G\to\dbU^{[n,p]}$, let $\rho^*\colon H^2(\dbU^{[n,p]})\to H^2(G)$ be the induced homomorphism on the cohomology.
We write $\rho^*\gam_{n,s}$ for the pullback of $\gam_{n,s}$ to $H^2(G)$ under $\rho$.
We define $H^2(G)_{n,s}$ to be the subset of $H^2(G)$ consisting of all such pulbacks.
\end{defin}

\begin{rems}
\label{properties of rho star}
\rm
(1) \quad
The pullback $\rho^*(\gam_{n,s})\in H^2(G)$ corresponds to the central extension \cite{Ledet05}*{p.\ 33}
\begin{equation}
\label{fiber product}
0\to\dbZ/p\to  \dbU\times_{\dbU^{[n,p]}}G\to G\to1,
\end{equation}
where the fiber product is with respect to the projection $\dbU\to\dbU^{[n,p]}$ and $\rho$.

\smallskip

\noindent
(2) \quad
In particular, $\rho^*(\gam_{n,s})=0$ if and only if (\ref{fiber product}) splits, which means that the embedding problem
\begin{equation}
\label{EP}
\xymatrix{
&&&G\ar@{-->}[dl]_{\widehat\rho}\ar[d]^{\rho}&\\
0\ar[r]&\dbZ/p\ar[r]&\dbU\ar[r]&\dbU^{[n,p]}\ar[r]&1
}
\end{equation}
is solvable, i.e., there  is continuous homomorphism $\hat\rho$ making it commutative.
Compare \cite{Hoechsmann68}*{1.1}.
\end{rems}

\begin{exam}
\label{example Bockstein}
\textsl{Bockstein elements.} \
\rm
The Bockstein map
\[
\Bock_{p^{n-1},G}\colon H^1(G,\dbZ/p^{n-1})\to H^2(G)
\]
is the connecting homomorphism arising from the short exact sequence of trivial $G$-modules
\[
0\to\dbZ/p\to\dbZ/p^n\to\dbZ/p^{n-1}\to0.
\]
Now take $s=1$, so $R=\dbZ/p^n$, $\dbU=\dbU_2(\dbZ/p^n)\isom\dbZ/p^n$ via the map $\pr_{1,2}$, and $\dbU^{[n,p]}\isom\dbZ/p^{n-1}$ (see Lemma \ref{filtration of U}(a)).
Hence  the continuous homomorphisms $\rho\colon G\to\dbU^{[n,p]}$ may be identified with the elements of $H^1(G,\dbZ/p^{n-1})$;
namely, $\rho$ corresponds  to $\pr_{1,2}\circ\rho\in H^1(G,\dbZ/p^{n-1})$.
By \cite{Efrat17}*{Example 7.4(1)}, one has
\[
\rho^*\gam_{n,1}=\Bock_{p^{n-1},G}(\pr_{1,2}\circ\rho).
\]
Thus
\[
H^2(G)_{n,1}=\Img(\Bock_{p^{n-1},G}).
\]
\end{exam}

\begin{exam}
\label{example Massey}
\textsl{Massey products.} \
\rm
Suppose that $n=s\geq2$, so $\dbU=\dbU_n(\dbZ/p)$.
The subgroup $H^2(G)_{n,n}$ turns out to be related to the $n$-fold Massey product $\langle\cdot\nek\cdot\rangle \colon H^1(G)^n\to H^2(G)$.
For a review on Massey products in the context of profinite cohomology we refer e.g., to \cite{Efrat14}, which follows \cite{Dwyer75}*{\S2} in the discrete setting.
We recall that for $\chi_1\nek\chi_n\in H^1(G)$, the Massey product $\langle\chi_1\nek\chi_n\rangle$ is a \textsl{subset} of $H^2(G)$, with the following properties:
\begin{enumerate}
\item[(i)]
The product  $\langle\chi_1,\nek,\chi_n\rangle$ consists exactly of the pullbacks $\rho^*\gam_{n,n}$, where $\rho\colon G\to \dbU^{[n,n]}$ is a continuous homomorphism such that $\chi_i=\pr_{i,i+1}\circ\rho$, $i=1,2\nek n$ \cite{Efrat17}*{Example 7.4(2)}.
\item[(ii)]
In particular,  one has $\langle\chi_1,\nek,\chi_n\rangle\neq\emptyset$ if and only if
there is a continuous homomorphism $\rho\colon G\to\dbU^{[n,p]}$ such that $\chi_i=\pr_{i,i+1}\circ\rho$, $i=1\nek n$.
\item[(iii)]
In view of (i) and Remark \ref{properties of rho star}(2), one has $0\in\langle\chi_1,\nek,\chi_n\rangle$ if and only if there is a continuous homomorphism $\hat\rho\colon G\to\dbU$ such that $\chi_i=\pr_{i,i+1}\circ
\hat\rho$, $i=1\nek n$.
\item[(iv)]
When $n=2$ and  $\rho\colon G\to\dbU^{[2,p]}$ is the unique homomorphism such that $\chi_1=\pr_{1,2}\circ\rho$ and $\chi_2=\pr_{2,3}\circ\rho$, one has $\rho^*\gam_{2,2}=\chi_1\cup\chi_2$.
\item[(v)]
When $n=3$, if $\langle\chi_1,\chi_2,\chi_3\rangle$ is nonempty, then it is a coset in $H^2(G)$ of its \textsl{indeterminicity}
\[
\chi_1\cup H^1(G)+\chi_3\cup H^1(G)=\bigl\{\chi_1\cup\psi_1+\chi_3\cup\psi_3\ \bigm|\ \psi_1,\psi_3\in H^1(G)\bigr\}.
\]
\end{enumerate}

From (i) we immediately deduce that $H^2(G)_{n,n}$ is the union of all $n$-fold Massey products in $H^2(G)$.

In particular, take $n=s=2$.
By (iv) above,
\[
H^2(G)_{2,2}=H^1(G)\cup H^1(G)=\bigl\{\chi_1\cup\chi_2\ \bigm|\ \chi_1,\chi_2\in H^1(G)\bigr\}.
\]
\end{exam}

\section{Examples}
\label{section on examples}
%We now compute the subsets $H^2(G)_{n,s}$ for some specific groups $G$.
%
First we recall from \cite{FriedJarden08}*{\S17.4} the following terminology and facts on free profinite groups.
Let $G$ be a profinite group and $X$ a set.
A map $f\colon X\to G$ \textsl{converges to $1$}, if for every open normal subgroup $N$ of $G$, the set $X\setminus f\inv(N)$ is finite.

We say that a profinite group $S$ is a \textsl{free profinite group on basis $X$} with respect to a map $\iota\colon X\to S$ if
\begin{enumerate}
\item[(i)]
$\iota\colon X\to S$ converges to $1$ and $\iota(X)$ generates $S$ as a profinite group;
\item[(ii)]
For every profinite group $G$ and a map $f\colon X\to G$ converging to $1$, there is a unique continuous homomorphism $\hat f\colon S\to G$ such that $f=\hat f\circ\iota$ on $X$.
\end{enumerate}
A free profinite group $S=S_X$ on $X$ exists, and is unique up to a continuous isomorphism.
Moreover, $\iota$ is then injective, and one may identify $X$ with its image in $S$.

\begin{exam}
\rm
Let $S$ be a free profinite group, let $m>n$, and let $G=S^{[m,p]}$.
We show that then $\rho^*=0$ for every continuous homomorphism $\rho\colon G\to\dbU^{[n,p]}$, so in particular $\rho^*\gam_{n,s}=0$ for every $1\leq s\leq n$.
Since $H^2(S)=0$, the  five term exact sequence in profinite cohomology \cite{NeukirchSchmidtWingberg}*{Prop.\ 1.6.7} implies that the transgression map $\trg\colon H^1(S^{(n,p)})^S\to H^2(S^{[n,p]})$ is surjective.
The functoriality of $\trg$ gives rise to a commutative square
\[
\xymatrix{
H^1(S^{(n,p)})^S\ar[d]_{\res}\ar[r]^{\trg}&H^2(S^{[n,p]})\ar[d]^{\inf}\\
H^1(S^{(m,p)})^S\ar[r]^{\trg}& H^2(S^{[m,p]}).\\
}
\]
(see \cite{EfratMinac11}*{\S2B}).
By \cite{CheboluEfratMinac12}*{Lemma 5.4}, the map $\res\colon H^1(S^{(n,p)})^S\to H^1(S^{(n+1,p)})^S$ is trivial, and therefore so is the  left vertical map.
It follows that the right vertical inflation map is also trivial.

Now let $\rho\colon S^{[m,p]}\to\dbU^{[n,p]}$ be a continuous homomorphism.
The projectivity of $S$ yields a continuous homomorphism $\hat\rho\colon S\to\dbU$ such that the left square in the following diagram commutes:
\[
\xymatrix{
S\ar[r]\ar[d]_{\widehat\rho}& S^{[m,p]}\ar[d]^{\rho}\ar[r]& S^{[n,p]}\ar@{-->}[dl]^{\bar\rho}\\
\dbU\ar[r]&\dbU^{[n,p]}&\qquad.
}
\]
Now $\hat\rho(S^{(n,p)})\leq\dbU^{(n,p)}$, so $\hat\rho(S^{(n,p)})$ is mapped trivially into $\dbU^{[n,p]}$.
It follows that $\rho$  factors via a homomorphism $\bar\rho\colon S^{[n,p]}\to\dbU^{[n,p]}$, as in the diagram.
Then $\rho^*=\inf\circ\bar\rho^*$.
Since $\inf$ is trivial,  $\rho^*=0$.
\end{exam}

\begin{exam}
\label{cyclic example}
\rm
Let $G=\hat\dbZ^{[n,p]}=\dbZ/p^{n-1}$.
As this is a one-relator $p$-group,  one has $\dim_{\dbF_p}H^2(G)=1$ \cite{NeukirchSchmidtWingberg}*{Cor.\ 3.9.5}.
Take $2\leq s\leq n$.
When  $p=2$ we assume further that $3\leq s\leq n$.
An elementary calculation shows that then $n-s+1+\lfloor\log_ps\rfloor\leq n-1$.
Hence, by Proposition \ref{exponent}, the exponent of $\dbU=\dbU_s(\dbZ/p^{n-s+1})$ divides $p^{n-1}$.
Therefore the embedding problem (\ref{EP}) is solvable for every homomorphism $\rho\colon G\to \dbU^{[n,p]}$.
By Remark \ref{properties of rho star}(2), $\rho^*\gam_{n,s}=0$.

By contrast, when $s=1$ we have $\dbU\isom\dbZ/p^n$, $\dbU^{[n,p]}\isom\dbZ/p^{n-1}$.
Taking $\rho$ to be an isomorphism, we see that (\ref{EP}) is not solvable, so $\rho^*\gam_{n,1}\neq0$.

Finally,  suppose that $p=n=2$, so $G=\dbZ/2$.
Then $\gam_{2,2}$ corresponds to the short exact sequence
\[
0\to\dbZ/2\to D_4\to(\dbZ/2)^2\to0,
\]
where $D_4$ is the dihedral group of order $8$ and the kernel  is its unique normal subgroup of order $2$.
Let $\tau$ be an element of order $4$ in $D_4$, and $\bar\tau$ its image in $(\dbZ/2)^2$.
Define a homomorphism $\rho\colon \dbZ/2\to(\dbZ/2)^2$ by $\rho(\bar1)=\bar\tau$.
Then (\ref{EP}) is not solvable, whence $\rho^*\gam_{2,2}\neq0$.
\end{exam}

\section{The indecomposable quotient}
Let $R$ be a unital commutative ring, and let $A=\bigoplus_{s=0}^\infty A_s$ be a graded $R$-algebra.
Let $A_+=\bigoplus_{s=1}^\infty A_s$.
The proof of the following lemma is straightforward.

\begin{lem}
The following $R$-submodules of $A$ coincide:
\begin{enumerate}
\item[(i)]
The submodule of $A$ generated by all products $aa'$, where $a,a'\in A_+$;
\item[(ii)]
The submodule of $A$ generated by all products $aa'$, where $a,a'\in A_+$ are homogenous;
\item[(iii)]
The submodule of $A$ generated by all products $aa'$, where $a,a'$ are  elements of a given set $T$ of homogenous elements in $A_+$ which generates it as an $R$-module.
\end{enumerate}
\end{lem}
We denote this $R$-submodule by $D$, and call it the submodule of \textsl{weakly decomposable elements} of $A$.
Since $D$ is generated by homogenous elements, the quotient $A_\indec=A/D$ has the structure of a graded $R$-module, which we call the \textsl{indecomposable quotient} of $A$.
Note that $D_1=\{0\}$, so the graded module morphism $A\to A_{\indec}$ is an isomorphism in degree $1$.

In particular, let $X$ be a set, let $X^*$ be as before the free monoid on $X$, and let $\bigoplus_{w\in X^*}Rw$ be the free $R$-module on $X^*$, graded by the length map on $X^*$.
It can be completed to unital associative graded algebras in the following two ways:

\smallskip

(1) \quad {\sl The free associative $R$-algebra $R\langle X\rangle$ over $X$}:
Here the product is the concatenation in $X^*$ (simply denoted by  $(u,v)\mapsto uv$) extended by $R$-linearity, and the unit element is the unit element $1$ of $X^*$, i.e., the empty word.
The embedding $R\to R\langle X\rangle$ induces an isomorphism $R\isom R\langle X\rangle_\indec$ of graded $R$-modules (both concentrated in degree $0$).

\smallskip

(2) \quad
{\sl The  shuffle $R$-algebra $\Sh_R(X)$ on $X$}:
Here the product is the shuffle product $(u,v)\mapsto u\sha v$ on $X^*$ (see the Introduction) extended by $R$-linearity, and the unit element is again $1\in X^*$.

\smallskip

Whereas $R\langle X\rangle_\indec$ has a trivial structure,  the structure of $\Sh_R(X)_\indec$ is deep, as we shall see in the sequel.

The above two constructions are dual in a bialgebra context -- see \cite{LodayVallette12}*{\S1.3.2} for details.

\section{Words}
\label{section on words}
We fix a nonempty set $X$, considered as an alphabet, and fix a total order on $X$.
Then the free monoid $X^*$ on $X$ is totally ordered with respect to the alphabetical order $\leq$ induced by the total order on $X$.
As before, we consider its elements as words in the alphabet $X$.
Let $X^s$ denote the set of all words $w$ of length $|w|=s$.
Thus $X^0=\{1\}$.

A nonempty word $w\in X^*$ is called a \textsl{Lyndon word} if  it is smaller with respect to $\leq$  than all its non-trivial proper right factors (i.e., suffixes).
We denote the set of all Lyndon words of length $s$ on $X$ by  $\Lyn_s(X)$.

The number of Lyndon words of length $s$ over $X$ can be expressed in terms of  {\sl Witt's necklace function}, defined for integers $s,m\geq1$ by
\[
\varphi_s(m)=\frac1s\sum_{d|s}\mu(d)m^{s/d}.
\]
Here $\mu$ is the M\"obius function, defined by $\mu(d)=(-1)^k$, if $d$ is a product of $k$ distinct prime numbers, and $\mu(d)=0$ otherwise.
We also define $\varphi_s(\infty)=\infty$.
Then one has \cite{Reutenauer93}*{Cor.\ 4.14}
\begin{equation}
\label{LyndonWitt}
|\Lyn_s(X)|=\varphi_s(|X|).
\end{equation}

Every word $w\in X^*$ can be uniquely  written as a concatenation $w=u_1^{i_1}\cdots u_k^{i_k}$ of Lyndon words $u_1>\cdots>u_k$, where $i_1\nek i_k\geq1$, and $u^i$ denotes the concatenation of $i$ copies of $u$ \cite{Reutenauer93}*{Th.\ 5.1 and Cor.\ 4.7}.
Note that then $|w|=i_1|u_1|+\cdots+i_k|u_k|$.
We define a non-associative polynomial
\[
Q_w=\frac1{i_1!\cdots i_k!}u_1^{\sha i_1}\sha\cdots\sha u_k^{\sha i_k}.
\]
in $\dbQ\langle X\rangle$, where $u^{\sha i}$ is the shuffle product of $i$ copies of $u$.

\begin{rems}
\label{rems on Qw}
\rm
\begin{enumerate}
\item[(1)]
The polynomial $Q_w$ is homogenous of the same degree $i_1|u_1|+\cdots+i_k|u_k|$ as $w$.
\item[(2)]
The coefficients of $Q_w$ are in fact non-negative integers \cite{Reutenauer93}*{p.\ 128}.
\item[(3)]
When $w$ is a Lyndon word, $Q_w=w$.
\item[(4)]
If $w$ is not a Lyndon word, then $i_1!\cdots i_k!Q_w$ is weakly decomposable in $\Sh_\dbZ(X)$.
\end{enumerate}
\end{rems}

We will need the following result by Radford \cite{Radford79} and Perrin and Viennot (unpublished); see  \cite{Reutenauer93}*{Th.\ 6.1}.

\begin{thm}
\label{Radford thm}
For every word $w$ one has $Q_w=w+\sum_{v<w}a_{v,w}v$ for some non-negative integers $a_{v,w}$.
\end{thm}
Note that in this sum we may restrict ourselves to words $v$ of the same length $s$ as $w$.

Assume for a moment that $X$ is finite, and let $[Q_w]_{w\in X^s}$ and $[w]_{w\in X^s}$ denote the column matrices with respect to the lexicographic order on $X^s$.
Theorem \ref{Radford thm} gives an upper-triangular unipotent matrix $M$ with entries in $\dbZ$ such that $[Q_w]_{w\in X^s}=M[w]_{w\in X^s}$.
Since $M$ is invertible over $\dbZ$, and  $X^s$ is a $\dbZ$-linear basis of the homogeneous component $\dbZ\langle X\rangle_s$ of  $\dbZ\langle X\rangle$ of degree $s$, the polynomials $(Q_w)_{w\in X^s}$ also form a $\dbZ$-linear basis of $\dbZ\langle X\rangle_s$.
Thus
\begin{equation}
\label{aternative basis}
\dbZ\langle X\rangle_s=\bigoplus_{u\in X^s}\dbZ u=\bigoplus_{w\in X^s}\dbZ Q_w,
\end{equation}
and the polynomials $(Q_w)_{w\in X^*}$  form a $\dbZ$-linear basis of $\dbZ\langle X\rangle$.
By a limit argument, the same holds also when $X$ is infinite.

For the following proposition, note that the structural graded module epimorphism $\Sh_\dbZ(X)\to\Sh_\dbZ(X)_\indec$ induces for every $s\geq0$  epimorphisms of $\dbZ$-modules
\[
\dbZ\langle X\rangle_s=\Sh_\dbZ(X)_s\to\Sh_\dbZ(X)_{\indec,s}\xrightarrow{-\tensor1}\Sh_\dbZ(X)_{\indec,s}\tensor(\dbZ/p)
\]
 on the $s$-homogenous components.

\begin{prop}
\label{Lyndon words generate}
Suppose that $1\leq s<p$.
Then  the images of the Lyndon words of length
 $s$ generate the $\dbZ/p$-module $\Sh_\dbZ(X)_{\indec,s}\tensor(\dbZ/p)$.
\end{prop}
\begin{proof}
Consider  $w\in X^s$ with its unique presentation $w=u_1^{i_1}\cdots u_k^{i_k}$ as a concatanation of Lyndon words as above.
Then $1\leq i_1\nek i_k\leq s<p$, so $i_1!\cdots i_k!$ is invertible modulo $p$.
If in addition $w\not\in\Lyn_s(X)$, then $i_1!\cdots i_k!Q_w$ is weakly indecomposable in $\Sh_\dbZ(X)$ (Remark \ref{rems on Qw}(4)).
Therefore in this case, $Q_w$ has a trivial image in $\Sh_\dbZ(X)_{\indec,s}\tensor(\dbZ/p)$.

Hence the $\dbZ$-modules  $\sum_{w\in\Lyn_s(X)}\dbZ Q_w$ and $\sum_{w\in X^s}\dbZ Q_w$ have the same image in $\Sh_\dbZ(X)_{\indec,s}\tensor(\dbZ/p)$.
By (\ref{aternative basis}),  this image is all of $\Sh_\dbZ(X)_{\indec,s}\tensor(\dbZ/p)$.
\end{proof}

\begin{exam}
\label{exam of Sh for a single letter}
\rm
In Proposition \ref{Lyndon words generate} one cannot omit the assumption that $s<p$.
Indeed, suppose that $X=\{x\}$ consists of a single letter.
Then for every $s$, the set $X^s$ consists solely of the concatenation power $(x)^s=(x\cdots x)$ ($s$ times).
Therefore $\Sh_\dbZ(X)_s=\dbZ\langle X\rangle_s\isom \dbZ$ for every $s$.

Let $D$ be the module of weakly decomposable elements in $\Sh_\dbZ(X)$.

In degree $1$ we have $D_1=\{0\}$, so $\Sh_\dbZ(X)_{\indec,1}=\Sh_\dbZ(X)_1\isom\dbZ$, whence $\Sh_\dbZ(X)_{\indec,1}\tensor(\dbZ/p)\isom\dbZ/p$.

In degree $s\geq2$ the $\dbZ$-module $D_s$ is generated by the shuffles $(x)^i\sha(x)^{s-i}=\binom si(x)^s$,  $i=1,2\nek s-1$.
Hence $D_s=\dbZ h(s)(x)^s$, where
\[
h(s)=\gcd\Bigl\{{\binom si}\ \Bigm|\ 1\leq i\leq s-1\Bigr\},
\]
implying that $\Sh_\dbZ(X)_{\indec,s}\isom \dbZ/h(s)$.
Now one has $p|h(s)$ if and only if $s$ is a $p$-power (Proposition \ref{binomial coef}(d)).
Therefore  $\Sh_\dbZ(X)_{\indec,s}\tensor(\dbZ/p)$ is $\dbZ/p$ if $s$ is a $p$-power, and is $0$ otherwise.

Moreover, we have $\Lyn_1(X)=\{(x)\}$ and $\Lyn_s(X)=\emptyset$ for $s\geq2$.
Therefore when $s>1$ is a $p$-power, the images of the Lyndon words of length $s$ do not generate $\Sh_\dbZ(X)_{\indec,s}\tensor(\dbZ/p)$.
In particular, the image of $\Lyn_p(X)$ does not generate $\Sh_\dbZ(X)_{\indec,p}\tensor(\dbZ/p)$.
\end{exam}

\section{Shuffle relations}
\label{section on shuffle relations}
Given a commutative unital ring $R$, let $R\langle\langle X\rangle\rangle$ be the ring of formal power series in the set $X$ of non-commuting variables and with coefficients from $R$.
Let $R\langle\langle X\rangle\rangle^\times$ be its group of invertible elements.
Note that for $x\in X$ one has $(1+x)\sum_{i=0}^\infty(-1)^ix^i=1$, so $1+x\in R\langle\langle X\rangle\rangle^\times$.

Assuming further that $R$ is finite, $R\langle\langle X\rangle\rangle^\times$  is a profinite group in a natural way \cite{Efrat14}*{\S5}.
Let $S$ be the free profinite group on the basis $X$ (see \S\ref{section on examples}).
There is a unique continuous homomorphism, $\Lam_R\colon S\to R\langle\langle X\rangle\rangle^\times$, called the \textsl{profinite Magnus homomorphism},  satisfying $\Lam_R(x)=1+x$ for every $x\in X$; see \cite{Efrat14}*{\S5} for its detailed construction.
This construction is functorial, in the sense that for every ring homomorphism $\varphi\colon R\to R'$ with an induced group homomorphism $\tilde\varphi\colon R\langle\langle X\rangle\rangle^\times\to R'\langle\langle X\rangle\rangle^\times$ one has $\tilde\varphi\circ\Lam_R=\Lam_{R'}$.

As in \cite{Efrat17}*{Example 7.4}, for every integers $n,s$ such that  $1\leq s\leq n$ and $n\geq2$, and every word $w=(x_1\cdots x_s)\in X^s$,  we associate a cohomology element $\alp_{w,n}\in H^2(S^{[n,p]})$ as follows:

Let again  $\dbU=\dbU_s(R)$ where $R=\dbZ/p^{n-s+1}$.
First we define a map $\rho_R^w\colon S\to \dbU_s(R)$ by setting the entry $(\rho_R^w(\sig))_{ij}$, where $i<j$, to be the coefficient of the subword $(x_i\cdots x_{j-1})$ in the power series $\Lam_R(\sig)$.
The map $\rho^w_R$ is a continuous group homomorphism \cite{Efrat14}*{Lemma 7.5}.
It induces a homomorphism $\bar\rho^w_R\colon S^{[n,p]}\to\dbU^{[n,p]}$.
Let $(\bar\rho^w_R)^*\colon H^2(\dbU^{[n,p]})\to H^2(S^{[n,p]})$ be the induced homomorphism on the cohomology groups.
For $\gam_{n,s}$ as in \S\ref{section on unitriangular matrices}, we obtain the pullback
\[
\alp_{w,n}=(\bar\rho_R^w)^*(\gam_{n,s})\in H^2(S^{[n,p]}).
\]

Fixing a total order on $X$,  the main theorem of \cite{Efrat17}  states:

\begin{thm}
\label{Efr17 Main Theorem}
For $n\geq2$, the cohomology elements $\alp_{w,n}$, where $w$ ranges over all Lyndon words of length $1\leq|w|\leq n$ in the alphabet $X$, form an $\dbF_p$-linear basis of $H^2(S^{[n,p]})$.
\end{thm}

For $f\in \dbZ\langle X\rangle$ and $w\in X^*$, let $f_w$ be the coefficient of $w$ in $f$.
By \cite{Efrat17}*{Th.\ 9.4}, one has the following \textsl{shuffle relations}:

\begin{thm}
\label{shuffle relations}
Let $u,v\in X^*$ be nonempty words with $s=|u|+|v|\leq n$.
Then
\[
\sum_{w\in X^s}(u\sha v)_w\alp_{w,n}=0.
\]
\end{thm}

We now prove our main result, and at the same time strengthen Proposition \ref{Lyndon words generate}:

\begin{thm}
\label{main results}
Let $p$ be a prime number and $n$ an integer such that $2\leq n<p$.
Then:
\begin{enumerate}
\item[(a)]
The map $w\mapsto\alp_{w,n}$ induces a $\dbZ/p$-module isomorphism
\[
\Bigl(\bigoplus_{s=1}^n\Sh_\dbZ(X)_{\indec,s}\Bigr)\tensor(\dbZ/p)\xrightarrow{\sim} H^2(S^{[n,p]}).
\]
\item[(b)]
For every $1\leq s<p$,  the images of the Lyndon words of length $s$ form a linear basis of the $\dbZ/p$-module $\Sh_\dbZ(X)_{\indec,s}\tensor(\dbZ/p)$.
\end{enumerate}
\end{thm}
\begin{proof}
For $1\leq s\leq n$, the map $X^s\to H^2(S^{[n,p]})$, $w\mapsto \alp_{w,n}$, extends by linearity to a $\dbZ$-module homomorphism
\[
\Phi_s\colon \dbZ\langle X\rangle_s=\bigoplus_{w\in X^s}\dbZ w\to H^2(S^{[n,p]}).
\]
Thus for $f\in \dbZ\langle X\rangle_s$ one has
\[
\Phi_s(f)=\Phi_s\bigl(\sum_{w\in X^s}f_ww\bigr)=\sum_{w\in X^s}f_w\alp_{w,n}.
\]
By Theorem \ref{shuffle relations}, $\Phi_s(u\sha v)=0$ for any nonempty words $u,v\in X^*$ with $s=|u|+|v|$.
Consequently, $\Phi_s$ factors via $\Sh_\dbZ(X)_{\indec,s}$.
Therefore it induces a $\dbF_p$-linear map
\begin{equation}
\label{epimorphism}
\bar\Phi_s\colon \Sh_\dbZ(X)_{\indec,s}\tensor(\dbZ/p)\to H^2(S^{[n,p]}).
\end{equation}

By Proposition \ref{Lyndon words generate}, for every $s$ the image of $\Lyn_s(X)$  spans $\Sh_\dbZ(X)_{\indec,s}\tensor(\dbZ/p)$.
Furthermore, by Theorem \ref{Efr17 Main Theorem}, the cohomology elements $\bar\Phi_s(\bar w)$, where $1\leq s\leq n$ and $\bar w\in L_s$, form an $\dbF_p$-linear basis of $H^2(S^{[n,p]})$.
It follows that $\bigoplus_{s=1}^n\bar\Phi_s\colon(\bigoplus_{s=1}^n\Sh_\dbZ(X)_{\indec,s}\tensor(\dbZ/p)\to H^2(S^{[n,p]})$ is an isomorphism, and the image of $\Lyn_s(X)$ in facts form a basis of $\Sh_\dbZ(X)_{\indec,s}\tensor(\dbZ/p)$ for every $1\leq s\leq n$.
\end{proof}

From Theorem \ref{main results}(b) and (\ref{LyndonWitt}) we deduce:

\begin{cor}
For $1\leq s<p$ one has
\[
\dim_{\dbF_p}(\Sh_\dbZ(X)_{\indec,s}\tensor(\dbZ/p))=\varphi_s(|X|).
\]
\end{cor}

The next example shows that in part (a) of Theorem \ref{main results} one cannot omit the assumption that $n<p$
 (for part (b) this has already been shown in Example \ref{exam of Sh for a single letter}).

\begin{exam}
\label{counterexamples}
\rm
Suppose that $X=\{x\}$.
Then $S^{[n,p]}=\hat\dbZ^{[n,p]}\isom\dbZ/p^{n-1}$, whence $\dim_{\dbF_p}H^2(S^{[n,p]})=1$.
On the other hand, Example \ref{exam of Sh for a single letter} shows that
\[
\dim_{\dbF_p}\Sh_\dbZ(X)_{\indec,1}\tensor(\dbZ/p)=\dim_{\dbF_p}\Sh_\dbZ(X)_{\indec,p}\tensor(\dbZ/p)=1,
\]
so there  is no isomorphism as in Theorem \ref{main results}(a).

Explicitly, Example \ref{cyclic example} shows that when $p\neq2$ and $n\geq p$, one has $\alp_{w,n}=(\bar\rho^w_{\dbZ/p^{n-p+1}})^*\gam_{n,p}=0$ for words $w$ of length $p$.
Thus $\bar\Phi_p$ is trivial on the nonzero summand $\Sh_\dbZ(X)_{\indec,p}\tensor(\dbZ/p)$.

Now suppose that $p=n=2$.
We have $\Bock_{2,G}(\chi)=\chi\cup\chi$ for every profinite group $G$ and every $\chi\in H^1(G)$ \cite{EfratMinac11}*{Lemma 2.4}.
Also, using the functoriality of the Magnus homomorphism we get $\pr_{1,2}\circ\bar\rho_{\dbZ/4}^{(x)}=\pr_{1,2}\circ\rho_{\dbZ/2}^{(xx)}=\pr_{3,4}\circ\rho_{\dbZ/2}^{(xx)}$.
Combining this with Examples \ref{example Bockstein} and \ref{example Massey} we obtain that
\[
\begin{split}
\bar\Phi_1(\overline{(x)})&=(\rho_{\dbZ/4}^{(x)})^*\gam_{2,1}
=\Bock_{2,\dbZ/2}(\pr_{1,2}\circ\bar\rho_{\dbZ/4}^{(x)})\\
&=(\pr_{1,2}\circ\bar\rho_{\dbZ/4}^{(x)})\cup(\pr_{1,2}\circ\bar\rho_{\dbZ/4}^{(x)})
=(\pr_{1,2}\circ\rho_{\dbZ/2}^{(xx)})\cup(\pr_{3,4}\circ\rho_{\dbZ/2}^{(xx)})\\
&=(\rho_{\dbZ/2}^{(xx)})^*\gam_{2,2}=\bar\Phi_2(\overline{(xx)}).
\end{split}
\]
This shows that the map $\bigoplus_{s=1}^n\bar\Phi_s$ is not injective in this case as well.
\end{exam}

\section{Absolute Galois groups}
\label{section on Galois groups}
Let $F$ be a field containing a root of unity of order $p$, and let $G=G_F$ be the absolute Galois group of $F$.
The Merkurjev--Suslin theorem implies that every element of $H^2(G)$ decomposes as a sum $\sum_i\chi_i\cup\chi'_i$ of cup products, $\chi_i,\chi'_i\in H^1(G)$.
In particular, for every continuous homomorphism $\rho\colon G\to\dbU(\dbZ/p^{n-s+1})^{[n,p]}$ the pullback $\rho^*\gam_{n,s}$ has such a decomposition.
While the Merkurjev--Suslin theorem is in general non-explicit, in a few cases the decomposition pattern of the  $\rho^*\gam_{n,s}$ is known, and in fact is related to some rather deep facts in Galois cohomology and to structural results on absolute Galois groups.

\medskip
\begin{exam}
\rm
Let $n=2$,  $s=1$.
Here $\dbU=\dbU_1(\dbZ/p^2)\isom\dbZ/p^2$ and $\dbU^{[2,p]}\isom\dbZ/p$ via $\pr_{1,2}$.
Denote the group of $p$th roots of unity by $\mu_p$.
Fix an isomorphism $i\colon\mu_p\xrightarrow{\sim}\dbZ/p$, let $i_*\colon H^1(G,\mu_p)\to H^1(G)=H^1(G,\dbZ/p)$ be the induced isomorphism,  and let $\zeta=i\inv(\bar1)$.
The {\sl Kummer map} $\kappa\colon F^\times=H^0(G,F_{\rm sep}^\times)\to H^1(G,\mu_p)$ is the connecting homomorphism arising from the short exact sequence
\[
1\to\mu_p\to F_\sep^\times\xrightarrow{p} F_\sep^\times\to1
\]
of discrete $G$-modules.
By \cite{EfratMinac11}*{Prop.\ 2.6},
\[
\Bock_{p,G}\cup\id=\id\cup\kappa
\]
as maps $H^1(G)\tensor H^0(G,\mu_p)\to H^2(G,\mu_p)$.
By the functoriality of the cup product \cite{NeukirchSchmidtWingberg}*{Prop.\ 1.4.2}, this can be restated as
\begin{equation}
\label{cohom identity for Bock}
\Bock_{p,G}\cup \id=\id\cup(i_*\circ\kappa\circ i\inv)
\end{equation}
as maps $H^1(G)\tensor H^0(G)\to H^2(G)$.
This cohomological identity was a main ingredient in the restriction on the group-theoretic structure of absolute Galois groups given in \cite{EfratMinac11}*{Main Theorem}.
It can be expressed in terms of  the pulbacks $\rho^*\gam_{2,1}$;
namely, for every continuous homomorphism $\rho\colon G\to\dbU^{[2,p]}$, Example \ref{example Bockstein} and (\ref{cohom identity for Bock}) (applied on $(\pr_{1,2}\circ\rho)\tensor\bar1$) give
\[
\rho^*\gam_{2,1}=\Bock_{p,G}(\pr_{1,2}\circ\rho)=(\pr_{1,2}\circ\rho)\cup i_*(\kappa(\zeta)).
\]
Therefore, in the notation of Definition \ref{basic definitions},
\[
H^2(G)_{2,1}=H^1(G)\cup i_*(\kappa(\zeta)).
\]

\end{exam}

\begin{exam}
\rm
When $n=s=2$, we generally have $\rho^*\gam_{2,2}=(\pr_{1,2}\circ\rho)\cup(\pr_{2,3}\circ\rho)$, by Example \ref{example Massey}(iv).
\end{exam}

\begin{exam}
\rm
Let  $n=s=3$ and $\dbU=\dbU_3(\dbZ/p)$.
It was shown in \cite{Matzri14}, \cite{EfratMatzri17} and \cite{MinacTan16} (following earlier works by Hopkins, Wickelgren \cite{HopkinsWickelgren15}, Min\'a\v c and T\^an) that if a $3$-fold Massey product in $H^2(G)=H^2(G_F)$ is  nonempty, then it contains $0$.
This imposes another restriction on the possible group-theoretic structure of absolute Galois groups - see \cite{MinacTan17}*{\S7}.
The following proposition interprets this fact in terms of the pulbacks $\rho^*\gam_{3,3}$:
\end{exam}

\begin{prop}
Let $\chi_1,\chi_2\in H^1(G)$.
The set $\chi_1\cup H^1(G)+\chi_2\cup H^1(G)$ consists exactly of all the pullbacks $\rho^*\gam_{3,3}$, where $\rho\colon G\to\dbU^{[3,p]}$ is a continuous homomorphism such that $\chi_1=\pr_{1,2}\circ\rho$ and $\chi_2=\pr_{3,4}\circ\rho$.

Consequently, $H(G)_{3,3}$ is the set $H^1(G)\cup H^1(G)+H^1(G)\cup H^1(G)$ of all sums of two cup products.
\end{prop}
\begin{proof}
Let $\rho $ be as in the first assertion.
By Example \ref{example Massey}(i),  the Massey product $\langle\pr_{1,2}\circ\rho,\pr_{2,3}\circ\rho,\pr_{3,4}\circ\rho\rangle$ contains $\rho^*\gam_{3,3}$, and in particular, is nonempty.
By the above result, it contains $0$.
In view of Example \ref{example Massey}(iv), it coincides with its indeterminicity, and in particular
\[
\rho^*\gam_{3,3}\in (\pr_{1,2}\circ\rho)\cup H^1(G)+(\pr_{3,4}\circ\rho)\cup H^1(G)
=\chi_1\cup H^1(G)+\chi_2\cup H^1(G).
\]

For the converse,  note that since $\chi_1,\chi_2\colon G\to\dbZ/p$ are continuous homomorphisms, the map
\[
\rho'=\begin{bmatrix}
1&\chi_1&0&0\\
0&1&0&0\\
0&0&1&\chi_2\\
0&0&0&1
\end{bmatrix}\colon G\to\dbU
\]
is also a continuous homomorphism.
This shows that $0\in \langle\chi_1,0,\chi_2\rangle$ (Example  \ref{example Massey}(iii)), so again
\[
\langle\chi_1,0,\chi_2\rangle=\chi_1\cup H^1(G)+\chi_2\cup H^1(G).
\]
Hence, by Example  \ref{example Massey}(i), every element of $\chi_1\cup H^1(G)+\chi_2\cup H^1(G)$ has the form $\rho^*(\gam_{3,3})$ for some continuous  homomorphism $\rho\colon G\to\dbU^{[3,p]}$ such that $\chi_1=\pr_{1,2}\circ\rho$ and $\chi_2=\pr_{3,4}\circ\rho$.

\medskip

The second assertion of the proposition follows from the first one.
\end{proof}

\begin{rems}
\rm
(1) \quad
In the previous discussion one can replace $G_F$ by the maximal pro-$p$ Galois group $G_F(p)=\Gal(F(p)/F)$ of $F$, where $F(p)$ is the maximal pro-$p$ Galois extension of $F$.
Indeed, $\inf\colon H^i(G_F(p))\to H^i(G_F)$ is an isomorphism for every $i$ \cite{CheboluEfratMinac12}*{Remark 8.2}.
Moreover, there is a relative Kummer map $F^\times\to H^1(G_F(p),\mu_p)$ which is a connecting homomorphism arising from the exponentiation by $p$ map $F(p)^\times\to F(p)^\times$, and it commutes with $\kappa$ and the inflation map $H^1(G_F(p),\mu_p)\to H^1(G_F,\mu_p)$.

\smallskip

\noindent
(2) \quad
It will be interesting to find in the absolute Galois group case explicit decompositions of $\rho^*\gam_{n,s}$ as a sum of cup products for other values of $n,s$.
\end{rems}

\end{document}